\documentclass[12pt]{amsart}

   \usepackage[colorlinks, pagebackref, linkcolor=red, citecolor=blue]{hyperref}
\usepackage{amsmath,amssymb}
\usepackage{color}
\usepackage{mathrsfs}
\usepackage{graphicx}

\numberwithin{equation}{section}
\theoremstyle{plain}
\newtheorem{theorem}{Theorem}[section]
\newtheorem{lemma}[theorem]{Lemma}

\newtheorem{definition}{Definition}[section]

\newtheorem*{remark}{Remark}

\newtheorem{corollary}[theorem]{Corollary}
\newtheorem{proposition}[theorem]{Proposition}
\newtheorem{question}[theorem]{Questions}
\newtheorem{example}[theorem]{Example}

 {\theoremstyle{definition}\newtheorem{exemple}[theorem]{Example}}
 
\newcommand{\reff}[1]{(\ref{#1})}

\def\T{ \mathbb T}

\def\H{H^\infty}

\def\D{{ \mathbb D}}
\def\C{{ \mathbb C}}

\def\N{{ \mathbb N}}

\def\e{\varepsilon}

\def\dis{\displaystyle}
\def\union{\cup}
\def\Union{\bigcup}
\def\inter{\cap}

\def\ov{\overline}

\def\ss{\subseteq}
\def\emp{\emptyset}

\def\buildrel#1_#2^#3{\mathrel{\mathop{\kern 0pt#1}\limits_{#2}^{#3}}}

\def\BP{Blaschke product}

\def\IBP{interpolating Blaschke product}

\begin{document}

\title []{One-component inner functions}


 \author{Joseph Cima}
\address{\small Department of Mathematics,
UNC,
Chapel Hill, North-Carolina,USA}
\email{cima@email.unc.edu}

 \author{Raymond Mortini}
  \address{
Universit\'{e} de Lorraine\\
 D\'{e}partement de Math\'{e}matiques et  
Institut \'Elie Cartan de Lorraine,  UMR 7502\\
 Ile du Saulcy\\
 F-57045 Metz, France} 
 \email{raymond.mortini@univ-lorraine.fr}



\begin{abstract}
We explicitely unveil several classes of  inner functions $u$ in $\H$ with the property that
 there is $\eta\in ]0,1[$
such that the level set $\Omega_u(\eta):=\{z\in\D: |u(z)|<\eta\}$ is connected. 
These  so-called one-component inner functions play an important role in operator theory.
\end{abstract}

 \maketitle

\centerline{\small \sl Dedicated to the memory of Vadim Tolokonnikov}\medskip
\centerline {\small\the\day.\the \month.\the\year} \medskip
\section*{Introduction}

\begin{definition}
An inner function $u$ in $\H$ is said to be a
 {\it one-component inner function} if there is $\eta\in ]0,1[$
such that the level set (also called sublevel set or filled level set) $\Omega_u(\eta):=\{z\in\D: |u(z)|<\eta\}$ is connected.
\end{definition}

One-component inner functions, the collection of which we denote by $\mathfrak I_c$, were first studied by B. Cohn \cite{coh}  in connection with embedding theorems and Carleson-measures.
It was shown in \cite[p. 355] {coh} for instance that arclength on $\{z\in\D: |u(z)|=\e\}$ is such a measure whenever 
$$\Omega_u(\eta)=\{z\in\D: |u(z)|<\eta\}$$
is connected and $\eta<\e<1$.

 A thorough study of the class $\mathfrak I_c$ was given by  A.B. Aleksandrov \cite{alex} who showed the
 interesting result that $u\in \mathfrak I_c$ if and only if  there is a constant $C=C(u)$ such that for all $a\in \D$
 $$\sup_{z\in \D}\left| \frac{1-\ov {u(a)} u(z)}{1-\ov a z} \right| \leq C \frac{1-|u(a)|^2}{1-|a|^2}.$$

Many operator-theoretic applications are given in \cite{alex, almp, bes, bbk}. In our paper here we are interested
in explicit examples, which  are somewhat lacking in literature.  For example, if $S$ is the atomic inner function,
which is given by 

$$S(z)=\exp\left(- \frac{1+z}{1-z}\right),$$
then all level sets $\Omega_S(\eta)$, $0<\eta<1$ are connected, because these sets coincide with the disks
\begin{equation}\label{horo}
D_\eta:=\left\{z\in\D:\left |z- \frac{L}{L+1}\right|<\frac{1}{L+1}\right\},~ L:= \log \frac{1}{\eta},
\end{equation}
which are tangential  to the unit circle at $p=1$.

The scheme of our note here is as follows: in section \ref{levi} we prove a general result on level sets which will be the key for our approach to the problem of unveiling classes of one-component inner functions. Then in section \ref{geo} we first present with elementary geometric/function theoretic methods several examples and then we use Aleksandrov's criterion to achieve this goal.  
For instance, we prove that
$BS, B\circ S$ and $S\circ B$ are in $\mathfrak I_c$ whenever $B$ is a finite \BP. Considered are also \IBP s. 
It will further be  shown that, under the supremum norm, 
 $\mathfrak I_c$ is an open subset of the set of all inner functions and multiplicatively closed.
In the final section we give counterexamples. 

\section{Level sets}\label{levi}

We first begin with a topological property of the class of general level sets. Although 
 statement (1) is ``well-known" (the earliest appearance seems to be  in \cite[Theorem VIII,~ 31]{tsu}), we could nowhere locate  a proof.
  The argument that the result is a simple and direct consequence of the maximum principle is, in our viewpoint, not tenable.

\begin{lemma}\label{compos1}
Given a non-constant  inner function $u$ in $\H$ and $\eta\in \;]0,1[$, 
let  $\Omega:=\Omega_u(\eta)=\{z\in\D: |u(z)|<\eta\}$ be a level set. Suppose that $\Omega_0$ is a
 component (=maximal connected subset)  of $\Omega$.
 Then
\begin{enumerate}
\item[(1)]  $\Omega_0$  is a simply connected domain; that is, $\C\setminus \Omega_0$ has no bounded components 
\footnote{A shorter proof can be given by using the advanced definition that a  domain $G$ in $\C$ is simply connected if 
every curve in $G$ is contractible in $G$, or equivalently, if 
 for every Jordan curve $J$ in $G$ the interior of $J$ belongs to $G$. That depends though  on the Jordan curve theorem.}.

\item[(2)]  $\inf_{\Omega_0} |u|=0$.
\end{enumerate}
\end{lemma}
\begin{proof}

 We show that (1) holds for every holomorphic function $f$ in $\D$; that is if $\Omega_0$ is a component of the level set $\Omega_f(\eta)$, $\eta>0$, then it is a simply connected domain
  \footnote{ This proof, as well as two different ones, including the one mentioned in  footnote 1, stem from
 the forthcoming book  manuscript \cite{moru} of the second author together with R. Rupp.}.
Note that each component $\Omega_0$ of the open set $\Omega_f(\eta)$
 is  an   open subset of $\D$. We may assume that $\eta$ is chosen so  that  
 $\{z\in\D: |f(z)|=\eta\}\not=\emp$.

Suppose, to the contrary, that $D$ is a  bounded component of $\C\setminus \Omega_0$.
Note that $D$ is closed in $\C$.  Then, necessarily, $D$ is contained  in $\D$, because the unique
unbounded complementary component of $\Omega_0$ contains $\{z\in \C: |z|\geq 1\}$.
Hence $D$ is a compact subset of $\D$.  Let $G:=\Omega_0^*$ be    the simply-connected hull  of $\Omega_0$; that is the union of $\Omega_0$
 with all bounded complementary components of $\Omega_0$. Note that $G$ is open because it coincides with the complement of the 
 unique unbounded complementary component of $\Omega_0$.
 Then, by definition of the simply connected hull,  $D\ss G$.  Now if $H$ is  any  bounded complementary 
 component  of $\Omega_0$ then (as it was the case for $D$) $H$ is a compact subset of $\D$ and so $\partial H\ss \D$. Moreover, 
 \begin{equation}\label{hcompo}
 \partial H\ss\partial\Omega_0.
 \end{equation}  
 In fact, 
  given $z_0\in\partial H$, let $U$ be a disk centered at $z_0$. Then $U\inter \Omega_0\not=\emp$,
since otherwise $U\union H$ would be a connected set strictly bigger than $H$ and contained in 
the complement of $\Omega_0$; a contradiction to the maximality of $H$. 
Since $z_0\in \partial H\ss H\ss \C\setminus \Omega_0$,
we conclude that $z_0\in \partial \Omega_0$.
  
  Now  $\partial H\ss\partial\Omega_0$ and $\Omega_0\ss \Omega_f(\eta)$ imply that 
    $|f|\leq \eta$ on $\partial H$, and so, by the maximum principle,
  $|f|\leq \eta$ on $H$. Consequently, 
 $|f|\leq \eta$ on $G$.
 By the local maximum principle, $|f|<\eta$ on $G$. 
  Since  $\partial D\ss D\ss G$,
 \begin{equation}\label{klein-auf-rand}
\mbox{$  |f|<\eta$ on $\partial D$}.
\end{equation}
On the other hand,
\begin{equation}\label{rand-zu-rand}
\partial D\buildrel\ss_{}^{\reff{hcompo}} \partial \Omega_0\inter \D\ss \{z\in \D: |f(z)|=\eta\}.
\end{equation}
Note that the second inclusion follows from the fact that  if $|f(z_0)|<\eta$ for $z_0\in \partial \Omega_0\inter \D$, then $\Omega_0$ would no longer be   a maximal connected subset of $\Omega_f(\eta)$.
Hence  $|f|=\eta$ on $\partial D$.  
This is a contradiction to \reff{klein-auf-rand}. Thus $\Omega_0$ is a simply connected domain.

(2) If $\ov\Omega_0\ss \D$, then, due to $\partial \Omega_0\ss \{z\in \D: |u(z)|=\eta\}$,
we obtain from the minimum principle that $u$ must have a zero in $\Omega_0$. 
Now let  $E:=\ov \Omega_0\inter \partial\D\not=\emp$. In view of achieving a contradiction, 
suppose that $u$ is bounded away from zero in $\Omega_0$.
Then $1/|u|$ is subharmonic and bounded  in $\Omega_0$ and  
$$\limsup_{\xi\to x\atop x\in \partial \Omega_0\setminus E}|u(\xi)|^{-1}=\eta^{-1}.$$
Since $u$ is an inner function, $E$ has linear measure zero (by \cite[Theorem 4.2]{ber}).
The maximum principle for subharmonic functions with  few exceptional points (here on the  set $E$; see \cite{be-co} or  \cite{gar}), 
  now implies that 
$|u|^{-1}\leq \eta^{-1}$ on $\Omega_0$. But $|u|<\eta$ on  $\Omega$ is a contradiction.
We conclude that $\inf_{\Omega_0}|u|=0$.

\end{proof}

\begin{lemma}\cite{coh}\label{com-bi}
Let $u$ be an inner function.
Then the connectedness of $\Omega_u(\eta)$ implies the one of $\Omega_u(\eta')$
for every $\eta'>\eta$.
\end{lemma}
\begin{proof}
Because $\Omega_u(\eta)$ is connected and  $\Omega_u(\eta)\ss \Omega_u(\eta')$, $\Omega_u(\eta)$ is contained in a 
unique component  $U_1(\eta')$ of $\Omega_u(\eta')$.
Suppose that $U_0(\eta')$ is a second  component of $\Omega_u(\eta')$. Then $|u|\geq \eta$
on $U_0(\eta')$, because  $U_0(\eta')$ is disjoint with $U_1(\eta')$ and hence with $\Omega_u(\eta)$.
By Lemma \ref{compos1} though, $\inf_{U_0(\eta')} |u|=0$; a contradiction.
Thus $\Omega_u(\eta')$ is connected.
\end{proof}

\section{Explicit examples of one-component inner functions}\label{geo}

 Let 
 $$\rho(z,w)=\left|\frac{z-w}{1-\ov zw}\right|$$ be the pseudohyperbolic distance of $z$ to $w$ in $\D$ and 
$$D_\rho(z_0,r)=\{z\in\D: \rho(z,z_0)<r\}$$
the associated disks, $0<r<1$.  Here is a first class of examples of functions in $\mathfrak I_c$.
Although the next Proposition  must be known (in view of A.B. Aleksandrov's criterion  \cite{alex}), see \ref{alexi}
below),
we include a simple geometric proof for the reader's convenience.

\begin{proposition} \label{fbp}
Let $B$ be  a finite Blaschke product. Then $B\in \mathfrak I_c$.
\end{proposition}
\begin{proof}
Denote by $z_1,\dots,z_N$ the zeros of $B$, multiplicities included. 
Let $\eta\in \;]0,1[$ be chosen so close to 1 that $G:=\Union_{n=1}^N D_\rho(z_n,\eta)$ is connected
(for example by  choosing $\eta$ so that $z_j\in D_\rho(z_1,\eta)$ for all $j$). Now
$$G\ss \{z\in \D: |B(z)|<\eta\}=\Omega_B(\eta),$$
 because $z\in G$ implies that for some $n$,
$$|B(z)|=\rho(B(z), B(z_n))\leq \rho(z,z_n)<\sigma.$$
Since $G$ is  connected, there is a unique component $\Omega_1$ of $\Omega$
containing $G$. 
 In particular, $Z(B)\ss G\ss \Omega_1$.
 If, in view of achieving a contradiction,  we suppose that $\Omega:=\Omega_B(\eta)$ is not connected,  there is a  component $\Omega_0$ of $\Omega$ which is disjoint with $\Omega_1$, and so with $G$. 
In particular, 
\begin{equation}\label{consta}
\rho(z, Z(B))\geq \sigma\;\; \text{for every}\;\;
z\in \Omega_0.
\end{equation}
Since $\ov \Omega_0 \ss \ov{\Omega_B(\eta)}\ss \D$, and $|B|=\eta$ on $\partial \Omega_0$,
we deduce from the minimum  principle that  $\Omega_0$ contains   a zero of $B$; a contradiction.
\end{proof}

We now generalize this result to a class of \IBP s.
Recall that a \BP\ $b$  with zero set/sequence $\{z_n:n\in \N\}$ is said to be an
 \IBP\ if $\delta(b):=\inf (1-|z_n|^2)|b'(z_n)|>0$.
If $b$ is an \IBP\  then, for small $\e$, the  pseudohyperbolic disks 
$$
D_\rho(z_n,r)=\{z\in\D: \rho(z,z_n)<\e\}
$$
are pairwise  disjoint.   Moreover,  by Hoffman's Lemma  (see below and also \cite{kl}), for any $\eta\in ]0,1[$, 
 $b$ is bounded away from zero on 
$\{z\in\D: \rho(z,Z(b))\geq \eta\}$.

\begin{theorem}[Hoffman's Lemma]\label{lemm}
 (\cite{ho} p. 86, 106 and \cite{ga} p. 404, 310).
Let $\delta, \eta$ and $\epsilon$ be real numbers, called Hoffman constants, 
satisfying $0 <~\delta <~1$, $ 0 < \eta < (1-\sqrt {1-\delta ^2})/ \delta $,
(that is, $0 < \eta < \rho(\delta,\eta)$) and

$$0 < \e < \eta \frac{\delta - \eta}{1- \delta \eta}.$$

If $B$ is any \IBP\ with zeros $\{z_n:n\in \N\}$ such that

$$\delta(B) = \inf_{n \in \Bbb N} (1- | z_n |^2) | B'(z_n)| \geq \delta,$$

then
\begin{enumerate}

\item[1)] the pseudohyperbolic disks $D_\rho(a,\eta)$
 for $a\in Z(B)$ are pairwise disjoint.
\item[(2)] The following inclusions hold:
$$\{z \in \D:  |B(z)| < \e \} \ss \{z \in \D: \rho(z, Z(B)) < \eta\} \ss \{z \in \D:  |B(z)| < \eta \}.$$

\end{enumerate}

\end{theorem}

  A large class of  \IBP s  which are one-component inner functions  now is given
in the following result.

\begin{theorem}\label{one-co}
Let $b$ be an \IBP\ with zero set $\{z_n:n\in \N\}$. Suppose that for some $\sigma\in \;]0,1[$ the set
$$G:=\Union_n D_\rho(z_n,\sigma)$$
is connected.
 Then $b$ is a one-component inner function. This holds in particular, if
 $\rho(z_n,z_{n+1})<\sigma<1$ for all $n$; for example if $z_n=1-2^{-n}$.

\end{theorem}
\begin{proof}
As in the proof of Proposition \ref{fbp}
 $$G\ss \{z\in \D: |b(z)|<\sigma\}=:\Omega.$$

Since $G$ is assumed to be connected, there is a unique component $\Omega_1$ of $\Omega$
containing $G$. 
 In particular, $Z(b)\ss G\ss \Omega_1$.
Now, if we suppose that $\Omega$ is not connected, then there is a  component $\Omega_0$
of $\Omega$ which is disjoint with $\Omega_1$, and so with $G$. 
In particular, 
\begin{equation}\label{consti}
\rho(z, Z(b))\geq \sigma\;\; \text{for every}\;\;
z\in \Omega_0.
\end{equation}
 Let  $\delta:=\delta(b)$,
$$ 0 < \eta < \min\{(1-\sqrt {1-\delta ^2})/ \delta, \sigma\},$$ 

$$0 < \e < \eta \frac{\delta - \eta}{1- \delta \eta}.$$
 
By Lemma \ref{compos1}, $\inf_{\Omega_0}|b|=0$. Thus, there is 
 $z_0\in \Omega_0$ be so that  $|b(z_0)|<\e$. We deduce from Hoffman's
 Lemma \ref{lemm} that  $\rho(z_0,Z(b))<\eta<\sigma$.
This is  a contradiction to \reff{consti}.
We conclude that $\Omega$ must be connected.  It is clear that the condition $\rho(z_n,z_{n+1})<\sigma$
for every $n$ implies that $\Union_n D_\rho(z_n,\sigma)$ is connected. For the rest, just note that
$$\rho(1-2^{-n}, 1-2^{-n-1})=\frac{2^{-n}-2^{-n-1}}{2^{-n}+2^{-n-1}+2^{-n}2^{-n-1}}=\frac{1}{3 +2^{-n}}.
$$
\end{proof}

\begin{corollary}\label{radi}
Let $B$ be a Blaschke product with increasing real zeros $x_n$ satisfying 
$$0<\eta_1\leq\rho(x_n, x_{n+1})\leq \eta_2<1.$$
Then $b\in \mathfrak I_c$.
\end{corollary}
\begin{proof}
Just use Theorem \ref{one-co} and the fact that by the Vinogradov-Hayman-Newman theorem, $B$ is interpolating if and only if
$$\sup_n\frac{1-x_{n+1}}{1-x_n}\leq s<1$$
or equivalently
$$\inf_n\rho(x_n,x_{n+1})\geq r>0.$$
\end{proof}

Using a result of Kam-Fook Tse  \cite{tse}, telling us that a sequence $(z_n)$  of points contained in  a Stolz angle (or cone)
$\{z\in \D: |1-z|< C (1-|z|)\}$ is interpolating if and only if it is separated 
(meaning that $\inf_{n\not =m}\rho(z_n,z_m)>0$),
we obtain:

\begin{corollary}\label{kft}
Let $B$ be a Blaschke product whose zeros $(z_n)$ are contained in a Stolz angle and are separated. Suppose that
$\rho(z_n,z_{n+1})\leq \eta<1$. Then $B\in \mathfrak I_c$. 
\end{corollary}

Similarily, using  a result by M. Weiss \cite[Theorem 3.6]{we} and its refinement in 
\cite[Theorem B]{bor}, we also obtain the following 
 assertion for sequences that may be tangential at 1 (see  also  Wortman \cite{wo}).

\begin{corollary}\label{we-wo}
Let $B$ be a Blaschke product  whose zeros $z_n=r_ne^{i\theta_n}$ satisfy: 
$$r_n<r_{n+1}, ~ \theta_{n+1}< \theta_n,$$
$$r_n\nearrow 1,~ \theta_n\searrow 0,$$
\begin{equation}\label{wort}
0<\eta_1\leq \rho(z_n,z_{n+1})\leq \eta_2<1.
\end{equation}
Then $B$ is an \IBP\ contained in $\mathfrak I_c$.
This holds in paticular if the zeros are located on a convex curve in $\D$ with endpoint $1$
and satisfying \reff{wort}.
\end{corollary}

Other classes of this type can be deduced from \cite{gw}.
Here are two explicit examples  of \IBP s  in $\mathfrak I_c$ whose zeros are given by iteration of the zero of  a hyperbolic, respectively parabolic automorphism of $\D$.  These functions appear, for instance, in the context of isometries on the Hardy space $H^p$ (see \cite{cw}). 

\begin{exemple}
$\bullet$~~ Let $\dis \varphi(z)=\frac{z-1/2}{1- (1/2) z}$. Then $\varphi$ is an hyperbolic automorphism
with fixed points $\pm 1$. If $\varphi_j:=\underbrace{\varphi\circ\cdots\circ \varphi}_{j-{\rm times}}$,
then $\varphi_j\in {\rm Aut}(\D)$ and vanishes exactly at the point
$$x_j:=\frac{3^j-1}{3^j+1}= 1-\frac{2}{3^j+1}.$$
This can readily be seen   by using that  $x_{j+1}=\varphi^{-1}(x_j)$ and 
$$\varphi_{j+1}(z) = (\varphi_ j\circ \varphi )(z)=\frac {z- \frac{\frac{1}{2}+x_j}
{1+ \frac{1}{2}x_j}}{1-z \frac{\frac{1}{2}+x_j}
{1+ \frac{1}{2}x_j}}.$$
Since 
$$\rho(x_j,x_{j+1})=\frac{3^{j+1}-3^j}{3^{j+1}+ 3^j}=\frac{1}{2},$$
we deduce from  Corollary \ref{radi} that the \BP\ 
$$B(z):=\prod_{j=1}^\infty \frac{x_j-z}{1-x_j z}$$
 associated with these zeros is in $\mathfrak I_c$.\\

$\bullet$ Let $\sigma\in {\rm Aut}(\D)$ and $\tau=\sigma\circ \varphi\circ \sigma^{-1}$. Then $\tau$ is also a hyperbolic automorphism
fixing the points $\sigma(\pm 1)$, and where $\xi:=\sigma(1)$ is the Denjoy-Wolff point with $\tau'(\xi)<1$.
 The zeros of the $n$-th iterate $\tau_n$ of $\tau$ are given by 
$$z_n=\tau_n^{-1}(0)=(\sigma\circ \varphi^{-1}_n\circ \sigma^{-1})(0).$$
 By the grand iteration theorem \cite[p.78]{sh},
since $1$ is an attracting fixpoint with $(\varphi^{-1})'(1)= 1/3<1$, the sequence $(\varphi^{-1}_n(\sigma^{-1}(0)))$ 
converges nontangentially to $1$.
Hence the points $z_n$ are located in  a cone with cusp at $\xi$. Moreover, if $n>k$ and  $a=\sigma^{-1}(0)$,
\begin{eqnarray*}\rho(z_n,z_k)&=& \rho\big( (\varphi^{-1}_n\circ \sigma^{-1})(0),  (\varphi^{-1}_k\circ \sigma^{-1})(0)\big)\\
&=& \rho\big(\varphi^{-1}_{n-k}(a), a\big)
\end{eqnarray*}
Thus,   $\rho(z_n,z_{n+1})=\rho (\varphi(a), a)$ for all $n$ and $\inf_{n\not=k}\rho(z_n,z_k)>0$. 
Now $(z_n)$ is a Blaschke sequence \footnote{ This also follows form the inequalities
$1-|\sigma(\xi_n)|^2= \frac{(1-|a|^2)(1-|\xi_n|^2)}{|1-\ov a \xi_n|^2}\leq \frac{1+|a|}{1-|a|} (1-|\xi_n|^2)$ and
 $1-|\psi_n(a)|^2\leq \frac{1+|a|}{1-|a|} (1-|w_n|^2)$, whenever $(w_n)$ is a Blaschke sequence and  $\psi_n(w_n)=\sigma(a)=0$.}
(\cite[Ex. 6, p. 85]{sh}); in fact, use d'Alembert's quotient criterion and  observe that  by the
Denjoy-Wolff theorem,
$$\frac{1-|z_{n+1}|}{1-|z_n|}= \frac{1-|\tau^{-1}(z_n)|}{1-|z_n|}\to (\tau^{-1})'(\xi)<1.$$
 Hence, by Corollary \ref{kft},
 $(z_n)$ is an interpolating sequence (see also \cite[p.80]{cm}) and the associated \BP\  $b=\prod_{n=1}^\infty e^{i\theta_n} \tau_n$
belongs to $\mathfrak I_c$ (here $\theta_n$ is chosen so that the $n$-th Blaschke factor is positive at the origin).

 \begin{figure}[h]
   \hspace{2,2cm}
  \scalebox{.40} 
  {\includegraphics{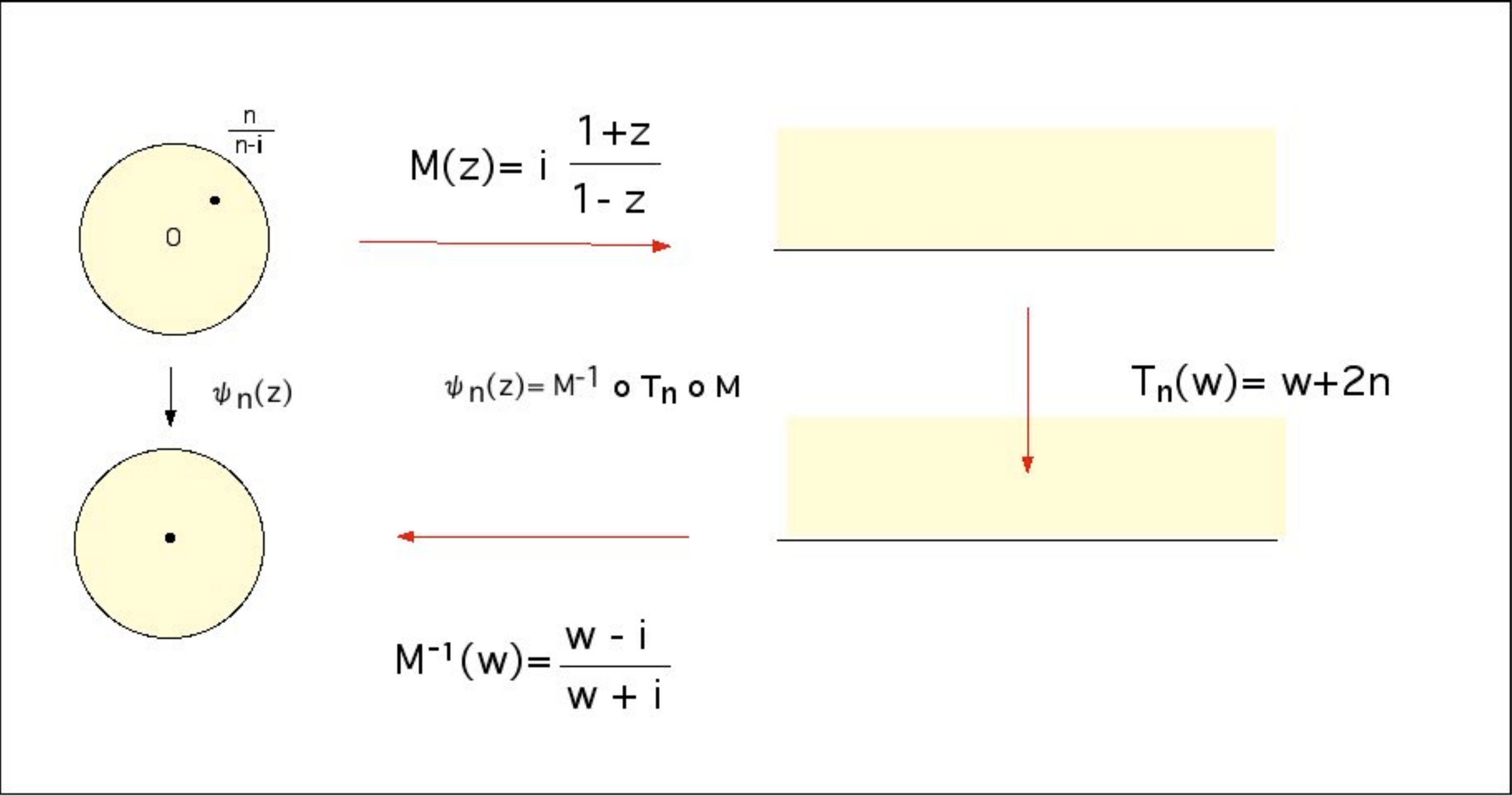}} 
\caption{\label{itera} {The parabolic automorphism}}
\end{figure}

$\bullet$~~ Let $\dis \psi(z)=i\; \frac{z- \frac{1+i}{2}}{1- \frac{1-i}{2}\; z}$. Then $\psi$ is a parabolic automorphism with attracting fixed point $1$. The automorphism $\psi$ is conjugated to the
translation $w\mapsto w+2$ on the upper half-plane (see figure \ref{itera}) via the map 
$M(z)= i (1+z)/(1-z)$ and $\psi_n=M^{-1}\circ T_n\circ M$.
The zeros of the $n$-th iterate 
$\psi_n$ of $\psi$ are given by 
$$z_n=\frac {n}{n-i};$$
just use that  $z_n=(M^{-1}\circ T_n^{-1}\circ M)(0)$.
These zeros satisfy  $\left|z_n-\frac{1}{2}\right|=\frac{1}{2}$ and of course also the Blaschke condition $\sum_{n=1}^\infty 1-|z_n|^2<\infty$.
Moreover,
$$\rho(z_n,z_{n+1})=\frac{1}{\sqrt{2}}.$$
Thus, by, Corollary \ref{we-wo},  the \BP\ associated with these zeros is  interpolating and belongs to $\mathfrak I_c$.\\
\end{exemple}

\begin{proposition}\label{ibp-mul}
Let $B$ be a finite \BP\ or an \IBP\ with real  zeros clustering at $p=1$. Then 
$f:=BS\in \mathfrak I_c$.
\end{proposition}
\begin{proof}
i) Let $B$ be  a finite \BP.  Chose $\eta\in \;]0,1[$ so close to 1 that the disk 
$D_\eta$  in \reff{horo}, which coincides with the level set $\Omega_S(\eta)$,  contains
all zeros of $B$.  Now $D_\eta=\Omega_S(\eta)\ss \Omega_f(\eta)$. Now $\Omega_f(\eta)$ must be connected,
since otherwise there would be a component $\Omega_0$ of $\Omega_f(\eta)$ disjoint from the component $\Omega_1$ containing $D_\eta$.  But $f$ is bounded away from zero outside 
$D_\eta$; hence $f=BS$ is bounded away from zero on $\Omega_0$.  This is a contradiction to
Lemma \ref{compos1} (2). 

ii) If $B$ is an \IBP\ with zeros $(z_n)$, then, by Hoffman's Lemma \ref{lemm}, $B$ is bounded away from zero outside  $R:=\Union D_\rho(z_n,\e)$ for every $\e\in \;]0,1[$. Now, if the zeros of $B$ are 
real, and bigger than $-\sigma$ for some $\sigma\in ]0,1[$,   this set $R$
is contained in a cone with cusp at 1 and  aperture-angle strictly less than $\pi$ (see for instance \cite{mru}).
Hence $R$ is contained 
in $D_\eta$ for all $\eta$ close to $1$. Thus, as above, we can deduce that $\Omega_{BS}(\eta)$
is connected.
\end{proof}

The previous result shows, in particular, that certain non one-component inner functions
(for example a thin \BP\ with positive zeros, see Corollary \ref{thin}), can be multiplied by a one-component inner function into $\mathfrak I_c$. In particular, $uv\in \mathfrak I_c$ does not imply that $u$ and $v$ belong to $\mathfrak I_c$.
 The reciprocal, though, is true: that  is $\mathfrak I_c$ itself is stable under multiplication, as we are going to show below.

\begin{proposition}\label{prodi}
Let $u,v$ be two inner functions in $\mathfrak I_c$. Then $uv\in\mathfrak I_c$.
\end{proposition}
\begin{proof}
Let $\Omega_u(\eta)$ and $\Omega_v(\eta')$ be two connected level sets. Due to  monotonicity
(Lemma \ref{com-bi}), and the fact that $\Union_{\lambda\in [\lambda_0,1[} \Omega_f (\lambda)=\D$,
we may assume that $\sigma$ satisfies
$$\max\{\eta,\eta'\}\leq \sigma<1$$
 and is so close to 1 that $0\in \Omega_u(\sigma)\inter \Omega_v(\sigma)\not=\emp$. Hence $U:=\Omega_u(\sigma)\union \Omega_v(\sigma)$ is connected. Now
$$\Omega_u(\sigma)\union \Omega_v(\sigma)\ss\Omega_{uv}(\sigma).$$
If we suppose that $\Omega_{uv}(\sigma)$ is not connected, then there is a component $\Omega_0$
of $\Omega_{uv}(\sigma)$ which is disjoint from $U$. In particular,
$u$ and $v$ are bounded away from zero on $\Omega_0$. This contradicts Lemma \ref{compos1} (2).
Hence $\Omega_{uv}(\sigma)$ is connected and so $uv\in \mathfrak I_c$.
\end{proof}

\begin{theorem}
The set of one-component  inner functions is open inside the set of all inner functions (with respect to the uniform norm topoplogy).
\end{theorem}
\begin{proof}
Let $u\in \mathfrak I_c$. Then, by Lemma \ref{com-bi},    $\Omega_u(\eta)$ is connected for all  $\eta\in  [\eta_0,1[$. Choose $0<\e<\min\{\eta, 1-\eta\}$ and let 
$\Theta$ be  an inner function with $||u-\Theta||<\e$. We claim that $\Theta\in\mathfrak I_c$, too.
To this end we note that
$$\Omega_{\Theta}(\eta-\e)\ss \Omega_u(\eta)\ss \Omega_\Theta(\eta+\e),$$
where $0<\eta-\e <\eta+\e<1$. As usual,
if we suppose that $\Omega_{\Theta}(\eta+\e)$ is not connected, then there is a component $\Omega_0$
of $\Omega_{\Theta}(\eta+\e)$ which is disjoint from the connected set $\Omega_u(\eta)$,
hence disjoint with  $\Omega_{\Theta}(\eta-\e)$. In other words, 
$|\Theta|\geq \eta-\e>0$ on  $\Omega_0$. This contradicts Lemma \ref{compos1} (2).
Hence $\Omega_\Theta(\eta+\e)$ is connected and so $\Theta\in\mathfrak I_c$.
\end{proof}

Next we look at right-compositions of $S$ with finite \BP s.  We first deal with the case where 
$B(z)=z^2$. 

\begin{example}
The function $S(z^2)$ is a one-component inner function.
\end{example}
\begin{proof}

 Let $\Omega_S(\eta)$ be the $\eta$-level set of $S$. Then, as we have already seen, this is a disk tangent to the unit circle at the point 1.
 We may choose $0<\eta<1$ so close to $1$ that $0$ belongs to $\Omega_S(\eta)$. Let $U=\Omega_S(\eta)\setminus ]-\infty, 0]$.
 Then $U$ is a simply connected slitted disk on which exists a holomorphic square root  $q$ of $z$. The image of $U$
 under $q$ is a simply connected domain $V$ in the semi-disk $\{z: |z|<1, {\rm Re}\; z>0\}$. Let $ V^*$ be its reflection along the 
 imginary axis.
 Then $E:=\ov{V^*\union V }$ is mapped by $z^2$ onto the closed disk  $\ov{\Omega_S(\eta)}$. 
 Then $E\setminus \partial E$ coincides with $\Omega_{S(z^2)}(\eta)$.
 
  \begin{figure}[h]
   \hspace{2,2cm}
  \scalebox{.40} 
  {\includegraphics{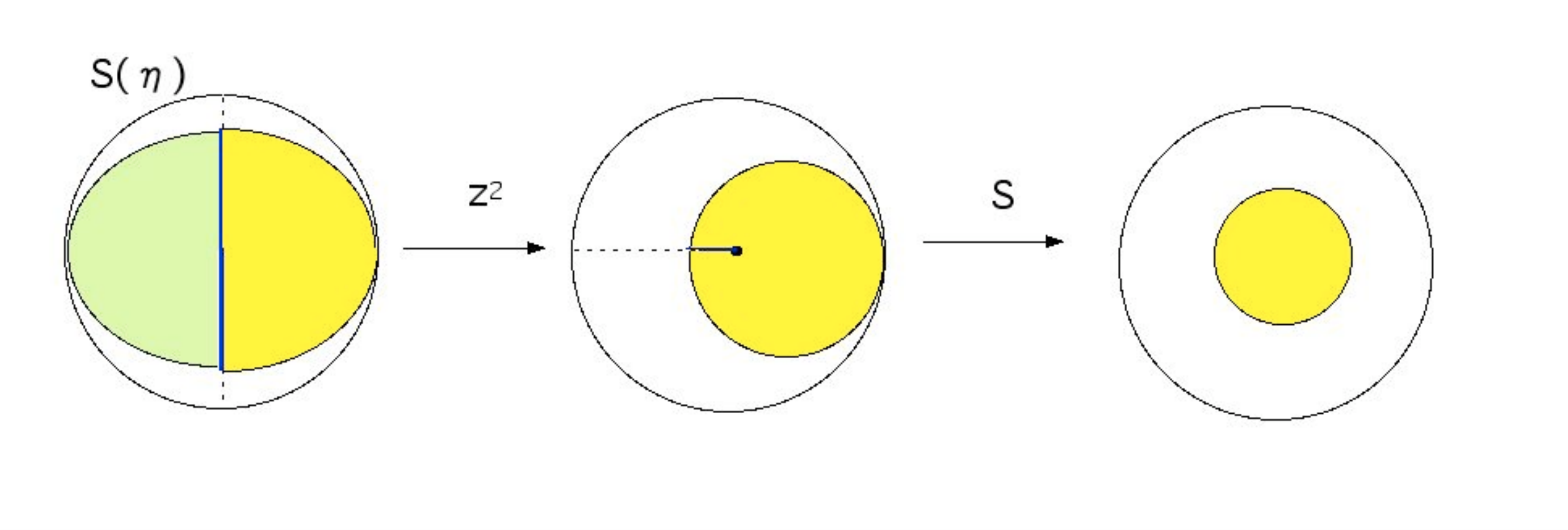}} 
\caption{\label{levels} {The level sets of $S(z^2)$}}
\end{figure}
\end{proof}

Using Aleksandrov's criterion (see below), we can extend this by replacing $z^2$ with any finite \BP.
Recall that the spectrum $\rho(\Theta)$ of an inner function  $\Theta$ is the set of all boundary points $\zeta$ for which 
$\Theta$ does not admit a holomorphic extension; or equivalently, for which  $Cl(\Theta,\zeta)=\ov\D$,
where 
$$\mbox{$Cl(\Theta,\zeta)=\{w\in\C: \exists (z_n)\in\D^\N,~ \lim z_n=\zeta$ and $ \lim \Theta(z_n)=w\}$}$$ 
is the cluster set of 
$\Theta$ at $\zeta$ (see \cite[p. 80]{ga}).

\begin{theorem}[Aleksandrov] \label{alexi}\cite[Theorem 1.11 and Remark 2, p. 2915]{alex}
Let $\Theta$ be an inner function. The following assertions are equivalent:
\begin{enumerate}
\item[(1)] $\Theta\in \mathfrak I_c$.
\item[(2)] There is a constant $C>0$ such that for every $\zeta\in \T\setminus \rho(\Theta)$ we have
$$i)~~~  |\Theta '' (\zeta)|\leq C\; |\Theta' (\zeta)|^2,$$
and 
$$ii)~~~ \mbox{$\liminf_{r\to 1} |\Theta(r\zeta)|<1$ for all $\zeta\in \rho(\Theta)$}.$$
\end{enumerate}
\end{theorem}
Note that, due to this theorem,  $\Theta\in \mathfrak I_c$ necessarily implies 
that $\rho(\Theta)$ has measure zero.

\begin{proposition}\label{sing}
Let $B$ be a finite \BP.  Then $S\circ B\in \mathfrak I_c$.
\end{proposition}
\begin{proof}
Let us note first that $\rho(S\circ B)=B^{-1}(\{1\})$. Since the derivative of $B$ on the boundary never vanishes (due to
\begin{equation}\label{deri}
z\frac{B'(z)}{B(z)}=\sum_{n=1}^N \frac{1-|a_n|^2}{|a_n-z|^2}, ~~|z|=1, B(a_n)=0,)
\end{equation} 
$B$ is schlicht in  a neighborhood of $1$. The angle conservation law now implies that  for 
$\zeta\in B^{-1}(1)$
the curve $r\mapsto B(r\zeta)$ stays in a Stolz angle at 1 in the image space of $B$.
Hence  $\liminf_{r\to 1} S(B(r \zeta))=0$ for $\zeta\in \rho(S\circ B)$. Now let us calculate the derivatives:
$$S'(z)=-S(z) \frac{2}{(1-z)^2},$$
$$S''(z)= S(z)\left[ \frac{4}{(1-z)^4}- \frac{4}{(1-z)^3}\right],$$

$$(S\circ B)'= (S'\circ B) B'$$

$$(S\circ B)'' =(S''\circ B) B'^2 + (S'\circ B) B''$$

\begin{eqnarray}\label{compos}
A:=\frac{(S\circ B)''}{[(S\circ B)']^2}&=&\frac{S''\circ B}{(S'\circ B)^2}+ \frac{(S'\circ B)}{(S'\circ B)^2}\frac{B''}{B'^2}\\\nonumber
&=&\frac{S''\circ B}{(S'\circ B)^2}+ \frac{1}{S'\circ B}\frac{B''}{B'^2}.
\end{eqnarray}
Hence, for $\zeta\in \T\setminus \rho(S\circ B)$, $|B(\zeta)|=1$ , but $\xi:=B(\zeta)\not=1$, and so,
by \reff{deri},
\begin{eqnarray*}
|A(\zeta)|&\leq& \sup_{\xi\not=1} \frac{|S''(\xi)|}{|S'(\xi)|^2}+ 2\sup_{\xi\not=1}\frac{|1-\xi|^2}{|S(\xi)|}\; C \\
&\leq& C' \sup_{\xi\not=1} \frac{|1-\xi|^4}{|1-\xi|^4} +8C<\infty.
\end{eqnarray*}\end{proof}

\begin{corollary}
Let $S_\mu$ be a singular inner function with finite spectrum $\rho(S_\mu)$. Then $S_\mu\in \mathfrak I_c$.
\end{corollary}
\begin{proof}
Since $S$ is the universal covering map of $\D\setminus \{0\}$, each singular inner function $S_\mu$
writes as $S_\mu=S\circ v$ for some inner function $v$. Since $\rho(S_\mu)$ is finite, $v$ necessarily is
a finite \BP. (This can also be seen from \cite[Proof of Theorem 2.2]{glmr}). The assertion now follows from Proposition \ref{sing}.
\end{proof}
Note that this result also follows in an elementary way from Proposition \ref{prodi} and the fact that every such $S_\mu$ is a finite product of powers of the atomic inner function $S$.
We now consider left-compositions with finite Blaschke products.

\begin{proposition} \label{frost}
Let $\Theta$ be a one-component inner function. Then each Frostman shift 
$(a-\Theta)/(1-\ov a \Theta)\in \mathfrak I_c$, too. Here $a\in \D$.
\end{proposition}
\begin{proof}
Let $\tau(z)=(a-z)/(1-\ov az)$.  Then $\rho(\tau\circ \Theta)=\rho(\Theta)$.
As above, $$\liminf_{r\to 1}|\tau\circ \Theta(r\zeta)|<1$$
 for every $\zeta\in \rho(\tau\circ\Theta)$.
Now
$$\tau(z)= \frac{1}{\ov a}+ \frac{|a|^2-1}{\ov a} \frac{1}{1-\ov a z},$$
from which we easily deduce the first and second derivatives.  By using  the  formulas \ref{compos},
we obtain
\begin{eqnarray*}
A:=\left|\frac{(\tau\circ \Theta)''}{[(\tau\circ \Theta)']^2}\right|&\leq &C\frac{|1-\ov a \Theta|^4}
{|1-\ov a\Theta|^3}+ C' |1-\ov a \Theta|^2 \; \frac{|\Theta''|}{|\Theta'|^2}.
\end{eqnarray*}
Hence, the assumption $\Theta\in \mathfrak I_c$ now yields (via Aleksandrov's criterion \ref{alexi})
that $\sup_{\zeta\in \rho(\tau\circ \Theta)} A(\zeta)<\infty$. Thus $\tau\circ\Theta\in \mathfrak I_c$.

\end{proof}
\begin{corollary}
Given $a\in\D\setminus \{0\}$, the interpolating Blaschke products\\ $(S-a)/(1-\ov a S)$ belong to $\mathfrak I_c$.
\end{corollary}

This also follows  from Corollary \ref{we-wo} by noticing that the $a$-points of $S$ are located on 
a disk tangent at $1$ and that the pseudohyperbolic distance between two consecutive ones is 
constant (see \cite{mo}). There it is also shown that the Frostman shift $(S-a)/(1-\ov a S)$ is an \IBP. 

\begin{corollary}
Let $B$ be  a finite Blaschke product and $\Theta\in \mathfrak I_c$. Then $B\circ \Theta\in \mathfrak I_c$.
\end{corollary}
\begin{proof}
This is  a combination of Propositions \ref{frost} and  \ref{prodi}.
\end{proof}

\section{Inner functions not belonging to $\mathfrak I_c$}\label{nonic}

Here we present a class of Blaschke products that are not one-component inner functions.
Recall that a Blaschke product $b$ with  zero-sequence $(z_n)$  is {\it thin} if 
$$\lim_n\prod_{k\not=n} \rho(z_k,z_n)=\lim_{n\to 1} (1-|z_n|^2)|b'(z_n)|=1.$$
It was shown by Tolokonnikov  \cite[Theorem 2.3]{to} that $b$ is thin if and only if
$$\lim_{|z|\to 1} (|b(z)|^2+(1-|z|^2)|b'(z)|)=1.$$

\begin{corollary}\label{thin}
Thin Blaschke products are never one-component inner functions.
\end{corollary}
\begin{proof}

Let $\e\in \;]0,1[$ be arbitrary close to $1$.  Choose $\eta>0$ and $\delta>0$ so close to 1 so that 
$$\mbox{$\e<\eta^2$ and $\dis \eta < (1-\sqrt {1-\delta ^2})/ \delta$}.$$
By deleting finitely many zeros, say $z_1,\dots, z_N$ of $b$, we obtain a tail $b_N$
such that $(1-|z_n|^2)|b_N'(z_n)|\geq \delta$ for every $n>N$.  Hence, by Theorem \ref{lemm},
\begin{equation}\label{bn}
\{z \in \D:  |b_N(z)| < \e \} \ss \{z \in \D: \rho(z, Z(b_N)) < \eta\}
\end{equation}
and the disks $D(z_n,\eta)$ are pairwise disjoint. This implies that  the level set 
$\{z \in \D:  |b_N(z)| < \e \}$
is not connected.  Now choose $r$ so close to 1 that 
$$p(z):=\prod_{n=1}^N \rho(z,z_n)\geq \e$$
for every $z$ with $r\leq |z|<1$.  We show that
the level set $\{|b|<\e^2\}$ is not connected. In fact,   for  some $r\leq |z|<1$ we have $|b(z)|<\e^2$, 
then
$$|b_N(z)|= \frac{|b(z)|}{|p(z)|}<\frac{\e^2}{\e}=\e.$$
Hence 
$$ \{z: r<|z|<1,  |b(z)|<\e^2\}\ss \{|b_N(z)|<\e\}\buildrel\ss_{}^{\reff{bn}}  \Union_{n>N} D(z_n,\eta).$$
Since the disks $D_\rho(z_n,\eta)$ are pairwise disjoint if $n>N$, we are done.
\end{proof}

\begin{corollary}\label{2bps}
No finite product $B$ of thin \IBP s belongs to $\mathfrak I_c$. 
\end{corollary}
\begin{proof}
Let $\e\in \;]0,1[$ be arbitrary close to 1. By Corollary \ref{thin}, if $b_j$, $(j=1,2)$, are 
two thin \BP s with  zero-sequence $(z_n^{(j)})_n$,
$$\Omega_{b_j}(\e)\ss \Union_{n=1}^\infty D_\rho(z^{(j)}_n,\eta)$$
for suitable $\eta$, the disks $D_\rho(z^{(j)}_n,\eta)$, being pairwise disjoint for $n$  large.
Since $\lim_n\rho(z_n^{(j)}, z_{n+1}^{(j)})=1$, we see that a disk $D_\rho(z^{(1)}_n,\eta)$
can meet at most one disk $D_\rho(z^{(2)}_m,\eta)$ for $n$ large. Hence 
$$\Omega_{b_1b_2}(\e^2)\ss \Union_{j=1}^2\Union_{n=1}^\infty D_\rho(z^{(j)}_n,\eta),$$
where the set on the right hand side obviously is disonnected.
The general case works via induction.
\end{proof}

\begin{remark}
The conditions 
\begin{equation}\label{no-one-co}
\eta^*:=\sup_{n\in \N}\rho(z_n, Z(b)\setminus \{z_n\})<1,
\end{equation}
or equivalently 
\begin{equation}\label{no-one-coco}
\mbox{$D(z_n,\eta)\inter \Union_{m\not= n}D(z_m, \eta)\not=\emp$ for some
$\eta\in ]0,1[$},
\end{equation}
are not sufficient to guarantee that the \IBP\ $b$ is a one-component inner function.
\end{remark}
\begin{proof}

Take $z_{2n}=1-n^{-n}$ and $z_{2n+1}=1-(n^{-n}+n^{-n})$.  Then $(z_{2n})$ and $(z_{2n+1})$ are (thin) interpolating sequences 
 by \cite[Corollary 2.4]{gm}. Using with $a=n^{-n}$ and $b=2a$ the identity
 $$\rho(1-a,1-b)=\frac{|a-b|}{a+b-ab},$$
 we conclude that
$$\rho(z_{2n}, z_{2n+1})= \frac{n^{-n}}{1-z_{2n} z_{2n+1}}\to 1/3,$$
and so
the union $(z_n)$ is an interpolating sequence satisfying \reff{no-one-coco}. By Corollary \ref{2bps},
the \BP\ formed with the zero-sequence $(z_n)$ is not in $\mathfrak I_c$.

\end{proof}

Using the following theorem in  \cite{ber},  we can exclude a much larger class of
Blaschke products from being one-component inner functions:
\begin{theorem}[Berman]
Let $u$ be an inner function. Then, for every $\e\in \;]0,1[$,  all the  components of the 
level sets $\{z\in \C: |u(z)|<\e\}$ have compact closures  in $\D$ if and only if $u$ is a Blaschke product and 
$$\mbox{$\limsup_{r\to 1} |u(r\xi)|=1$ for every $\xi\in \T$}.$$
\end{theorem}
In particular this condition is satisfied by finite products of thin Blaschke products (see 
\cite[Proposition 2.2]{gomo})
as well as by the class of uniform Frostman Blaschke products
$$\sup_{\xi\in \T} \sum_{n=1}^\infty \frac{1-|z_n|^2}{|\xi-z_n|}<\infty.$$
Note that this Frostman condition implies that the associated \BP\ has radial limits  of modulus
one everywhere \cite[p. 33]{cl}. As a byproduct of Theorem \ref{one-co} we therefore obtain
\begin{corollary}
If $b$  is a uniform Frostman Blaschke product with zeros $(z_n)$ clustering at a single point, 
then $\limsup_\rho(z_n,z_{n+1})=1$.
\end{corollary}

\begin{question}
{\rm To conclude, we would like to ask two questions and present three problems}:

\begin{enumerate}
\item [(1)] Can every inner function $u$ whose  boundary spectrum $\rho(u)$ has measure zero, be multiplied by a one-component inner function into $\mathfrak I_c$?
\item[(2)] Let $S_\mu$ be a singular inner function with countable spectrum. 
Give a characterization of those measures $\mu$ such that $S_\mu\in \mathfrak I_c$.
Do the same for singular continuous measures.
\item [(3)] In terms of the zeros, give a characterization of those \IBP s that belong to $\mathfrak I_c$.
\item[(4)]  Does the \BP\ $B$ with zeros $z_n=1-n^{-2}$ belong to $\mathfrak I_c$?
\end{enumerate}

\end{question}

\section{ Acknowledgement} 

We thank Rudolf Rupp and Robert Burckel for their valuable comments concerning 
Lemma \ref{compos1} (1),  the proof of which was orginally developed for the upcoming monograph \cite{moru}.


\end{document}